\documentclass[11pt]{amsart}

\usepackage{amssymb,amsfonts,bm,amsmath}
\allowdisplaybreaks

\begin{document}

\theoremstyle{plain}
\newtheorem{thm}{Theorem}[section]
\newtheorem{theorem}[thm]{Theorem}
\newtheorem{lemma}[thm]{Lemma}
\newtheorem{corollary}[thm]{Corollary}
\newtheorem{proposition}[thm]{Proposition}
\newtheorem{conjecture}[thm]{Conjecture}
\theoremstyle{definition}
\newtheorem{construction}[thm]{Construction}
\newtheorem{notations}[thm]{Notations}
\newtheorem{question}[thm]{Question}
\newtheorem{problem}[thm]{Problem}
\newtheorem{remark}[thm]{Remark}
\newtheorem{remarks}[thm]{Remarks}
\newtheorem{definition}[thm]{Definition}
\newtheorem{claim}[thm]{Claim}
\newtheorem{assumption}[thm]{Assumption}
\newtheorem{assumptions}[thm]{Assumptions}
\newtheorem{properties}[thm]{Properties}
\newtheorem{example}[thm]{Example}
\newtheorem{comments}[thm]{Comments}
\newtheorem{blank}[thm]{}
\newtheorem{observation}[thm]{Observation}
\newtheorem{defn-thm}[thm]{Definition-Theorem}

\newcommand{\sM}{{\mathcal M}}


\title[Recursion formulae of higher Weil-Petersson volumes]{Recursion formulae of higher Weil-Petersson volumes}
        \author{Kefeng Liu}
        \address{Center of Mathematical Sciences, Zhejiang University, Hangzhou, Zhejiang 310027, China;
                Department of Mathematics,University of California at Los Angeles,
                Los Angeles, CA 90095-1555, USA}
        \email{liu@math.ucla.edu, liu@cms.zju.edu.cn}
        \author{Hao Xu}
        \address{Center of Mathematical Sciences, Zhejiang University, Hangzhou, Zhejiang 310027, China}
        \email{haoxu@cms.zju.edu.cn}

        \begin{abstract}
   In this paper we study effective recursion formulae for computing intersection numbers
   of mixed $\psi$ and $\kappa$ classes on moduli spaces of curves. By using the celebrated Witten-Kontsevich theorem, we generalize Mulase-Safnuk form of Mirzakhani's
 recursion and prove a recursion formula of higher Weil-Petersson
 volumes. We also present recursion formulae to compute intersection pairings
 in the tautological rings of moduli spaces of curves.
        \end{abstract}
    \maketitle

\section{Introduction}

We denote by $\overline{\sM}_{g,n}$ the moduli space of stable
$n$-pointed genus $g$ complex algebraic curves. We have the morphism
that forgets the last marked point,
$$
\pi_{n+1}: \overline{\sM}_{g,n+1}\longrightarrow
\overline{\sM}_{g,n}.
$$
Denote by $\sigma_1,\dots,\sigma_n$ the canonical sections of $\pi$,
and by $D_1,\dots,D_n$ the corresponding divisors in
$\overline{\sM}_{g,n+1}$. Let $\omega_{\pi}$ be the relative
dualizing sheaf, we have the following tautological classes on
moduli spaces of curves.
\begin{align*}
\psi_i&=c_1(\sigma_i^*(\omega_{\pi}))\\
\kappa_i&=\pi_*\left(c_1\left(\omega_{\pi}\left(\sum D_i\right)\right)^{i+1}\right)\\
\lambda_l&=c_l(\pi_*(\omega_{\pi})),\quad 1\leq l\leq g.
\end{align*}
The classes $\kappa_i$ were first defined by Mumford \cite{Mu} on
$\overline{\sM}_g$. Their generalization to $\overline{\sM}_{g,n}$
here is due to Arbarello-Cornalba \cite{Ar-Co, Ar-Co2}. Before that
time, the classes $\kappa_i$ were defined  as $\pi_*(c_1
(\omega_\pi)^{i+1})$. Arbarello-Cornalba's definition turned out to
be the correct one especially from the point of view of the
restrictions to the boundary strata.

We are interested in the following intersection numbers
$$\langle\kappa_{b_1}\cdots\kappa_{b_k}\tau_{d_1}\cdots\tau_{d_n}\rangle_g:=
\int_{\overline{\sM}_{g,n}}\kappa_{b_1}\cdots\kappa_{b_k}\psi_1^{d_1}\cdots\psi_n^{d_n},$$
where $\sum b_j+\sum d_j=3g-3+n$. When $d_1=\cdots =d_n=0$, these
intersection numbers are called the higher Weil-Petersson volumes of
moduli spaces of curves.

The fact that intersection numbers involving both $\kappa$ classes
and $\psi$ classes can be reduced to intersection numbers involving
only $\psi$ classes was already known to Witten [9], and has been
developed by Arbarello-Cornalba \cite{Ar-Co}, Faber \cite{Fa} and
Kaufmann-Manin-Zagier \cite{KMZ} into a beautiful combinatorial
formalism. Faber has a wonderful maple program computing these
intersection numbers.

In a series of innovative papers \cite{Mir,Mir2}, Mirzakhani
obtained a beautiful recursion formula of the Weil-Petersson volumes
of the moduli spaces of bordered Riemann surfaces. As discussed by
Mulase and Safnuk in \cite{MS, Saf}, Mirzakhani's recursion formula
is equivalent to the following enlightening recursion relation of
intersection numbers.
\begin{multline*}
  (2k_1+1)!!\langle\kappa_1^{k_0}\tau_{k_1}\cdots\tau_{k_n}\rangle_g \\
  =  \sum_{j=2}^{n} \sum_{l=0}^{k_0}\frac{k_0!}{(k_0-l)!} \frac{(2(l+k_1+k_j)-1)!!}{(2k_j-1)!!} \beta_l \langle\kappa_1^{k_0-l}\tau_{k_1+k_j+l-1}\prod_{i \neq 1,j}\tau_{k_i}\rangle_g \\
 + \frac{1}{2} \sum_{l=0}^{k} \sum_{d_1 + d_2 = l+k_1 - 2}  \frac{k_0!}{(k_0-l)!}
  (2d_1+1)!! (2d_2+1)!! \beta_l \langle\kappa_1^{k_0-l}\tau_{d_1}\tau_{d_2}\prod_{i\neq 1}\tau_{k_i}\rangle_{g-1}  \\
  + \frac{1}{2}\sum_{\substack{m_0+n_0 = k_0-l \\ I \coprod J = \{2, \ldots, n \}}}
  \sum_{l=0}^{k_0}\sum_{d_1 + d_2 = l+k_1-2}  \frac{k_0!}{m_0! n_0!}
 (2d_1+1)!!(2d_2+1)!! \beta_l \\
  \times \langle\kappa_1^{m_0}\tau_{d_1}\prod_{i\in I} \tau_{k_i}\rangle_{g'}
 \langle\kappa_1^{n_0}\tau_{d_2}\prod_{i\in J} \tau_{k_i}\rangle_{g-g'},
\end{multline*}
where
\begin{equation*}
\beta_l = (2^{2l+1}-4)\frac{\zeta(2l)}{(2\pi^2)^l} =
(-1)^{l-1}2^l(2^{2l}-2) \frac{B_{2l}}{(2l)!}.
\end{equation*}

In a previous paper \cite{LX3}, it is shown that the
Witten-Kontsevich theorem implies the Mulase-Safnuk form of
Mirzakhani's recursion formula. Its relationship with matrix
integrals has been studied by Eynard and Orantin \cite{EO,Ey}.

More discussions about computations of Weil-Petersson or higher
Weil-Petersson volumes can be found in the papers \cite{Gr, KK, KMZ,
MZ, Pe, ST, Wo, Zo}.

Now we fix notation as in \cite{KMZ}. Consider the semigroup
$N^\infty$ of sequences ${\bold m}=(m(1),m(2),\dots)$ where $m(i)$
are nonnegative integers and $m(i)=0$ for sufficiently large $i$.
Denote by $\bm{\delta}_a$ the sequence with 1 at the $a$-th place
and zeros elsewhere.

Let $\bold m, \bold t, \bold{a_1,\dots,a_n} \in N^\infty$, $\bold
m=\sum_{i=1}^n \bold{a_i}$, and $\bold s:=(s_1,s_2,\dots)$ be a
family of independent formal variables.
$$|\bold m|:=\sum_{i\geq 1}i m(i),\quad ||\bold m||:=\sum_{i\geq1}m(i),\quad \bold s^{\bold m}:=\prod_{i\geq 1}s_i^{m(i)},\quad \bold m!:=\prod_{i\geq1}m(i)!,$$
$$\binom{\bold m}{\bold{t}}:=\prod_{i\geq1}\binom{ m(i)}{t(i)},\quad \binom{\bold m}{\bold{a_1,\dots,a_n}}:=\prod_{i\geq1}\binom{ m(i)}{a_1(i),\dots,a_n(i)}.$$

Let $\bold b\in N^\infty$, we denote a formal monomial of $\kappa$
classes by
$$\kappa(\bold b):=\prod_{i\geq1}\kappa_i^{b(i)}.$$

\begin{theorem} Let
$\bold b\in N^\infty$ and $d_j\geq 0$. Then
\begin{multline}
\sum_{\bold L+\bold{L'}=\bold b}(-1)^{||\bold L||}\binom{\bold
b}{\bold L}\frac{(2d_1+2|\bold L|+1)!!}{(2|\bold L|+1)!!}
\langle\kappa(\bold L')\tau_{d_1+|\bold L|}\prod_{j=2}^n\tau_{d_j}\rangle_g\\
=\sum_{j=2}^n\frac{(2(d_1+d_j)-1)!!}{(2d_j-1)!!}\langle\kappa(\bold{b})\tau_{d_1+d_j-1}\prod_{i\neq
1,j}\tau_{d_i}\rangle_g\\
+\frac{1}{2}\sum_{r+s=|d_1|-2}(2r+1)!!(2s+1)!!\langle\kappa(\bold{b})\tau_r\tau_s\prod_{i\neq1}\tau_{d_i}\rangle_{g-1}\\
+\frac{1}{2}\sum_{\substack{\bold{e}+\bold{f}=\bold b\\I\coprod
J=\{2,\dots,n\}}}\sum_{r+s=d_1-2}\binom{\bold b}{\bold{e}}(2r+1)!!(2s+1)!!\\
\times \langle\kappa(\bold{e})\tau_r\prod_{i\in
I}\tau_{d_i}\rangle_{g'}\langle\kappa(\bold{f})\tau_s\prod_{i\in
J}\tau_{d_i}\rangle_{g-g'}.
\end{multline}
\end{theorem}

\begin{theorem} Let
$\bold b\in N^\infty$ and $d_j\geq 0$. Then
\begin{multline}
(2d_1+1)!!\langle\kappa(\bold b)\tau_{d_1}\cdots\tau_{d_n}\rangle_g\\
=\sum_{j=2}^n\sum_{\bold L+\bold{L'}=\bold b}\alpha_{\bold
L}\binom{\bold b}{\bold L}\frac{(2(|\bold
L|+d_1+d_j)-1)!!}{(2d_j-1)!!}\langle\kappa(\bold{L'})\tau_{|\bold
L|+d_1+d_j-1}\prod_{i\neq
1,j}\tau_{d_i}\rangle_g\\
+\frac{1}{2}\sum_{\bold L+\bold{L'}=\bold b}\sum_{r+s=|\bold
L|+d_1-2}\alpha_{\bold
L}\binom{\bold b}{\bold L}(2r+1)!!(2s+1)!!\langle\kappa(\bold{L'})\tau_r\tau_s\prod_{i=2}^n\tau_{d_i}\rangle_{g-1}\\
+\frac{1}{2}\sum_{\substack{\bold L+\bold{e}+\bold{f}=\bold
b\\I\coprod J=\{2,\dots,n\}}}\sum_{r+s=|\bold
L|+d_1-2}\alpha_{\bold L}\binom{\bold b}{\bold
L,\bold{e},\bold{f}}(2r+1)!!(2s+1)!!\\
\times \langle\kappa(\bold{e})\tau_r\prod_{i\in
I}\tau_{d_i}\rangle_{g'}\langle\kappa(\bold{f})\tau_s\prod_{i\in
J}\tau_{d_i}\rangle_{g-g'},
\end{multline}
where the constants $\alpha_{\bold L}$ are determined recursively
from the following formula
$$\sum_{\bold L+\bold{L'}=\bold b}\frac{(-1)^{||\bold L||}\alpha_{\bold L}}{\bold L!\bold{L'}!(2|\bold{L'}|+1)!!}=0,\qquad \bold b\neq0,$$
namely
$$\alpha_{\bold b}=\bold b!\sum_{\substack{\bold L+\bold{L'}=\bold b\\ \bold{L'}\neq\bold 0}}\frac{(-1)^{||\bold L'||-1}\alpha_{\bold L}}{\bold L!\bold{L'}!(2|\bold{L'}|+1)!!},\qquad\bold b\neq0,$$
with the initial value $\alpha_{\bold 0}=1$.
\end{theorem}

Denote $\alpha(l,0,0,\dots)$ by $\alpha_l$, we recover
Mirzakhani's recursion formula with
$$\alpha_l=l!\beta_l=(-1)^{l-1}(2^{2l}-2)\frac{B_{2l}}{(2l-1)!!}.$$
We also have
$$\alpha(\bm{\delta}_l)=\frac{1}{(2l+1)!!}.$$ Setting
$\bold b=\bold 0$, we get the Witten-Kontsevich theorem \cite{Wi,Ko}
in the form of DVV recursion relation \cite{DVV}.

Note that Theorems 1.1 and 1.2 hold only for $n\geq 1$. If $n=0$,
i.e. for higher Weil-Petersson volumes of $\overline{\mathcal M}_g$,
we may apply the following formula first (see Proposition 3.1).

\begin{equation}
\langle\kappa(\bold b)\rangle_g=\frac{1}{2g-2}\sum_{\bold L+\bold
L'=\bold b}(-1)^{||\bold L||}\binom{\bold b}{\bold
L}\langle\tau_{|\bold L|+1}\kappa(\bold L')\rangle_g.
\end{equation}

So we can use Theorems 1.1 and 1.2 to compute any intersection
numbers of $\psi$ and $\kappa$ classes recursively with the three
initial values
$$\langle\tau_0\kappa_1\rangle_1=\frac{1}{24},\qquad
\langle\tau_0^3\rangle_0=1,\qquad
\langle\tau_1\rangle_1=\frac{1}{24}.$$ We have computed a table of
$\alpha_{\bold L}$ for all $|\bold L|\leq 15$ and have written a
maple program \cite{Maple} implementing Theorems 1.1 and 1.2.

In the arguments of Mirzakhani, Mulase and Safnuk, they use
Wolpert's formula \cite{Wo}
$$\kappa_1=\frac{1}{2\pi^2}\omega_{WP},$$
where $\omega_{WP}$ is the Weil-Petersson K\"ahler form. Since
Wolpert's formula has no counterpart for higher degree $\kappa$
classes,  there is no a priori reason that Theorem 1.2 shall be
true.

We are led to Theorem 1.2 also by the discovery that $\psi$ and
$\kappa$ classes are compatible, namely recursions of pure $\psi$
classes can be neatly generalized to recursions including both
$\psi$ and $\kappa$ classes, where $\kappa_1$ plays no special role.
This fact is equivalent to a relation of generating functions in
Theorem 4.4.

For $\bold b\in N^\infty$, we denote by $V_{g,n}(\kappa(\bold b))$
the higher Weil-Petersson volume $$\langle\tau_0^n \kappa(\bold
b)\rangle_g=\int_{\overline{\sM}_{g,n}}\kappa(\bold b).$$ Let
$V_g(\kappa(\bold b))$ denote $V_{g,0}(\kappa(\bold b))$.

Higher Weil-Petersson volumes were extensively studied in the paper
\cite{KMZ}. The authors found an explicit expression (see Lemma 2.2
below) of $V_{g,n}(\kappa(\bold b))$ in terms of integrals of $\psi$
classes. In genus zero, they obtained more nice results about
generating functions of $V_{0,n}(\kappa(\bold b))$ and raised the
question whether their methods may be generalized to higher genera.

Although we feel it is difficult to generalize
Kaufmann-Manin-Zagier's results to higher genera, we did find an
effective recursion formula between $V_{g,n}(\kappa(\bold b))$ valid
for all $g$ and $n$, based on our previous work on integrals of
$\psi$ classes. The results are contained in the following two
theorems.

\begin{theorem}
Let $\bold b\in N^\infty$ and $n\geq 1$. Then
\begin{multline}
\big(2g-1+||\bold b||\big)V_{g,n}(\kappa(\bold
b))=\frac{1}{12}V_{g-1,n+3}(\kappa(\bold b))-\sum_{\substack{\bold L+\bold{L'}=\bold b\\
||\bold{L'}||\geq 2}}\binom{\bold b}{\bold L}V_{g,n}(\kappa(\bold
L)\kappa_{|\bold{L'}|})\\
+\frac{1}{2}\sum_{\substack{\bold L+\bold{L'}=\bold b\\
\bold{L}\neq\bold 0, \bold{L'}\neq\bold
0}}\sum_{r+s=n-1}\binom{\bold b}{\bold
L}\binom{n-1}{r}V_{g',r+2}(\kappa(\bold L))V_{g-g',s+2}(\kappa(\bold
{L'})).
\end{multline}
\end{theorem}

Theorem 1.3 is an effective formula for computing higher
Weil-Petersson volumes recursively by induction on $g$ and $||\bold
b||$, with the initial values
$$V_{0,3}(1)=1 \quad \text{ and }\quad V_{0,n}(\kappa(\bm{\delta}_{n-3}))=1,\ n\geq 4,$$
where $\bm{\delta}_a$ denotes the sequence with $1$ at the $a$-th
place and zeros elsewhere.

\begin{theorem}
Let $g\geq2$ and $\bold b\in N^\infty$. Then
\begin{multline}
\big((2g-1)(2g-2)+(4g-3)||\bold b||+||\bold
b||^2\big)V_{g}(\kappa(\bold b))=5\sum_{\bold L+\bold{L'}=\bold
b}\binom{\bold b}{\bold L}V_{g,1}(\kappa(\bold
L)\kappa_{|\bold{L'}|+1})\\
-\frac{1}6{}\sum_{\bold L+\bold{L'}=\bold b}\binom{\bold b}{\bold
L}V_{g-1,3}(\kappa(\bold L)\kappa_{|\bold{L'}|})-\sum_{\bold
L+\bold{e}+\bold{f}=\bold b}\binom{\bold b}{\bold{
L,e,f}}V_{g',1}(\kappa_{|\bold L|}\kappa(\bold
e))V_{g-g',2}(\kappa(\bold
f))\\
-(2g-1+||\bold b||)\sum_{\substack{\bold L+\bold{L'}=\bold
b\\||\bold{L'}||\geq 2}}\binom{\bold b}{\bold L}V_{g}(\kappa(\bold
L)\kappa_{|\bold{L'}|})\\
-\sum_{\substack{\bold L+\bold{L'}=\bold b\\||\bold{L'}||\geq
2}}\binom{\bold b}{\bold L}\sum_{\bold e+\bold f=\bold
L+\bm{\delta}_{|\bold{L'}|}}\binom{\bold
L+\bm{\delta}_{|\bold{L'}|}}{\bold e} V_{g}(\kappa(\bold
e)\kappa_{|\bold f|}).
\end{multline}
\end{theorem}

By induction on $||\bold b||$, Theorem 1.4 reduces the computation
of $V_{g}(\kappa(\bold b))$ to the cases of $V_{g,n}(\kappa(\bold
b))$ for $n\geq1$, which have been computed by Theorem 1.3.
Therefore Theorems 1.3 and 1.4 completely determine higher
Weil-Petersson volumes of moduli spaces of curves.

The virtue of the above recursions is that they do not involve
$\psi$ classes. So if one wants to compute only higher
Weil-Petersson volumes, the above recursions are more efficient both
in speed and memory use, especially when we use ``option remember''
in a maple program.

On the other hand, we know that intersection numbers of mixed $\psi$
and $\kappa$ classes can be expressed by intersection numbers of
pure $\kappa$ classes \cite{Ar-Co}.

In Section 2, we prove Theorems 1.1 and 1.2. In Section 3 we prove
Theorems 1.3 and 1.4. In Section 4, we prove that the generating
functions of intersection numbers involving general $\kappa$ and
$\psi$ classes satisfy Virasoro constraints and the KdV hierarchy.
In Section 5, we consider recursions of Hodge integrals with
$\lambda$ classes.\\

\noindent{\bf Acknowledgements.} We would like to thank Chiu-Chu
Melissa Liu for helpful discussions. We thank the referees for
helpful suggestions.

\vskip 30pt
\section{Proofs of Theorems 1.1 and 1.2}

The following elementary lemma is crucial to our proof.
\begin{lemma}
Let $F(\bold L,n)$ and $G(\bold L,n)$ be two functions defined on
$N^\infty\times\mathbb N$, where $\mathbb N=\{0,1,2,\dots\}$ is the
set of nonnegative integers. Let $\alpha_{\bold L}$ and
$\beta_{\bold L}$ be real numbers depending only on $\bold L\in
N^\infty$ that satisfy $\alpha_{\bold 0}\beta_{\bold 0}=1$ and
$$\sum_{\bold L+\bold{L'}=\bold b}\alpha_{\bold L}\beta_{\bold{L'}}=0,\qquad\bold b\neq0.$$
Then the following two identities are equivalent.

\begin{align*}
G(\bold b,n)=\sum_{\bold L+\bold{L'}=\bold b} \alpha_{\bold L} F(\bold{L'},n+|\bold L|),\quad\forall\ (\bold b,n)\in N^\infty\times\mathbb N\\
F(\bold b,n)=\sum_{\bold L+\bold{L'}=\bold b} \beta_{\bold L} G(\bold{L'},n+|\bold L|),\quad\forall\ (\bold b,n)\in N^\infty\times\mathbb N\\
\end{align*}
\end{lemma}
\begin{proof}
Assume the first identity holds, we have
\begin{align*}
\sum_{\bold a=\bold 0}^{\bold b} \beta_{\bold a} G(\bold b-\bold a,n+|\bold a|)&=\sum_{\bold a=\bold 0}^{\bold b} \beta_{\bold a}
\sum_{\bold{a'}=\bold 0}^{\bold b-\bold a}\alpha_{\bold{a'}} F(\bold b-\bold a-\bold{a'},n+|\bold a+\bold{a'}|)\\
&=\sum_{\bold L=\bold 0}^{\bold b}\sum_{\bold a+\bold{a'}=\bold L}(\beta_{\bold a}\alpha_{\bold{a'}}) F(\bold b-\bold L,n+|\bold L|)\\
&=\sum_{\bold L=\bold 0}^{\bold b}\delta_{\bold L,\bold0}F(\bold b-\bold L,n+|\bold L|)\\
&=F(\bold b,n).
\end{align*}
So we have proved the second identity. The proof of the other
direction is the same.
\end{proof}

We also need the following combinatorial formula from \cite{KMZ}.
\begin{lemma}\cite{KMZ} Let $\bold m\in N^\infty$.
\begin{multline*}
\langle\prod_{j=1}^n\tau_{d_j}\kappa(\bold
m)\rangle_g=\sum_{k=0}^{||\bold m||}\frac{(-1)^{||\bold
m||-k}}{k!}\sum_{\substack {\bold
m=\bold{m_1}+\cdots+\bold{m_k}\\\bold{m_i}\neq\bold 0}}\binom{\bold
m}{\bold
{m_1,\dots,m_k}}\langle\prod_{j=1}^n\tau_{d_j}\prod_{j=1}^k\tau_{|\bold{m_j}|+1}\rangle_g\\
=\sum_{k\geq0}\sum_{\substack {\bold
m=a_1\bold{m_1}+\cdots+a_k\bold{m_k}\\\bold{m_i}\neq\bold {m_j},
i\neq j}}\frac{(-1)^{||\bold m||-\sum_{i=1}^k a_i}}{\prod_{i=1}^k
a_i!}\binom{\bold m}{ {\underbrace{\bold{
m_1,..,m_1}}_{a_1},\dots,\underbrace{\bold{m_k,..,m_k}}_{a_k}}}\langle\prod_{j=1}^n\tau_{d_j}\prod_{j=1}^k\tau_{|\bold{m_j}|+1}^{a_j}\rangle_g
\end{multline*}
where in the last term, these distinct $\{\bold{m_1,\dots,m_k}\}$
are unordered in the summation and $a_i$ are positive integers.
\end{lemma}
\begin{proof} We only give a sketch. Let
$ \pi_{n+p,n}:\overline{\sM}_{g,n+p}\longrightarrow
\overline{\sM}_{g,n} $ be the morphism which forgets the last $p$
marked points and denote
$\pi_{n+p,n*}(\psi_{n+1}^{a_1+1}\dots\psi_{n+p}^{a_p+1})$ by
$R(a_1,\dots,a_p)$, then we have the formula \cite{Ar-Co}
$$R(a_1,\dots,a_p)=\sum_{\sigma\in\mathbb S_p}\prod_{\substack{{\rm each\ cycle}\ c\\ {\rm of}\ \sigma}}\kappa_{\sum_{j\in c}a_j},$$
where we write any permutation $\sigma$ in the symmetric group
$\mathbb S_p$ as a product of disjoint cycles.

By a formal combinatorial argument, we get the following inversion
result
$$\kappa_{a_1}\cdots\kappa_{a_p}=\sum_{k=1}^p\frac{(-1)^{p-k}}{k!}\sum_{\substack{\{1,\dots,p\}=S_1\coprod\dots\coprod S_k\\S_k\neq\emptyset}}
R(\sum_{j\in S_1}a_j,\dots,\sum_{j\in S_k}a_j),$$ from which Lemma
2.2 follows.
\end{proof}

\bigskip
\noindent{\bf Proof of Theorem 1.1} \smallskip

Let LHS and RHS denote the left and right hand side of Theorem 1.1
respectively. By Lemma 2.2 and the Witten-Kontsevich theorem, we get
\begin{multline*}
(2d_1+1)!!\langle\prod_{j=1}^n\tau_{d_j}\kappa(\bold b)\rangle_g\\
=(2d_1+1)!!\sum_{k=0}^{||\bold b||}\frac{(-1)^{||\bold
b||-k}}{k!}\sum_{\substack {\bold{m_1}+\cdots+\bold{m_k}=\bold
b\\\bold {m_i}\neq0}}\binom{
\bold b}{\bold{m_1,\dots,m_k}}\langle\prod_{j=1}^n\tau_{d_j}\prod_{j=1}^k\tau_{|\bold{m_j}|+1}\rangle_g\\
=\sum_{k=0}^{||\bold b||} \frac{(-1)^{||\bold
b||-k}}{k!}\sum_{\substack {\bold{m_1}+\cdots+\bold{m_k}=\bold
b\\\bold{m_i}\neq0}}\binom{\bold
b}{\bold{m_1},\dots,\bold{m_k}}\qquad\qquad\qquad\qquad\qquad\qquad\qquad\\
\times\left(\sum_{j=2}^n\frac{(2(d_1+d_j)-1)!!}{(2d_j-1)!!}\langle\tau_{d_1+d_j-1}\prod_{i\neq
1,j}\tau_{d_i}\prod_{i=1}^k\tau_{|\bold{m_i}|+1}\rangle_g\right.\\
+\sum_{j=1}^k\frac{(2(d_1+|\bold{m_j}|)+1)!!}{(2|\bold{m_j}|+1)!!}\langle\tau_{d_1+|\bold{m_j}|}\prod_{i=2}^n\tau_{d_i}\prod_{i\neq
j}\tau_{|\bold {m_i}|+1}\rangle_g\\
+\frac{1}{2}\sum_{r+s=d_1-2}(2r+1)!!(2s+1)!!\langle\tau_r\tau_s\prod_{i=2}^n\tau_{d_i}\prod_{i=1}^k\tau_{|\bold{m_i}|+1}\rangle_{g-1}\\
\left.+\frac{1}{2}\sum_{\substack{I\coprod
J=\{2,\dots,n\}\\I'\coprod J'=\{1, \dots, k\}
}}\sum_{r+s=d_1-2}(2r+1)!!(2s+1)!!\right.\\
\left.\times\langle\tau_r\prod_{i\in I}\tau_{d_i}\prod_{i\in
I'}\tau_{|\bold{m_i}|+1}\rangle_{g'}\langle\tau_s\prod_{i\in
J}\tau_{d_i}\prod_{i\in
J'}\tau_{|\bold{m_i}|+1}\rangle_{g-g'}\right)\\
=\sum_{j=2}^n\frac{(2(d_1+d_j)-1)!!}{(2d_j-1)!!}\langle\kappa(\bold{b})\tau_{d_1+d_j-1}\prod_{i\neq
1,j}\tau_{d_i}\rangle_g\qquad\qquad\qquad\qquad\qquad\qquad\qquad\\
+\frac{1}{2}\sum_{r+s=|d_1|-2}(2r+1)!!(2s+1)!!\langle\kappa(\bold{b})\tau_r\tau_s\prod_{i\neq1}\tau_{d_i}\rangle_{g-1}\\
+\frac{1}{2}\sum_{\substack{\bold{e}+\bold{f}=\bold b\\I\coprod
J=\{2,\dots,n\}}}\sum_{r+s=d_1-2}\binom{\bold b}{\bold{e}}(2r+1)!!(2s+1)!!\\
\times \langle\kappa(\bold{e})\tau_r\prod_{i\in
I}\tau_{d_i}\rangle_{g'}\langle\kappa(\bold{f})\tau_s\prod_{i\in
J}\tau_{d_i}\rangle_{g-g'}\\
+\sum_{k=0}^{||\bold b||} \frac{(-1)^{||\bold
b||-k}}{k!}\sum_{\substack {\bold{m_1}+\cdots+\bold{m_k}=\bold
b\\\bold{m_i}\neq0}}\binom{\bold
b}{\bold{m_1},\dots,\bold{m_k}}\\
\times\sum_{j=1}^k\frac{(2(d_1+|\bold{m_j}|)+1)!!}{(2|\bold{m_j}|+1)!!}\langle\tau_{d_1+|\bold{m_j}|}\prod_{i=2}^n\tau_{d_i}\prod_{i\neq
j}\tau_{|\bold {m_i}|+1}\rangle_g\\
=RHS+\sum_{k\geq0}\frac{(-1)^{||\bold
b||-k-1}}{(k+1)!}\sum_{\substack{\bold L+\bold{L'} =\bold b\\\bold
L\neq\bold 0}}\sum_{\substack {\bold{m_1}+\cdots+\bold{m_k}=\bold
b-\bold L\\\bold{m_i}\neq0}}\binom{\bold b}{\bold L}\binom{\bold
b}{\bold{m_1},\dots,\bold{m_k}}\\
\times(k+1)\frac{(2(d_1+|\bold{L}|)+1)!!}{(2|\bold{L}|+1)!!}\langle\tau_{d_1+|\bold{L}|}\prod_{i=2}^n\tau_{d_i}
\prod_{i=1}^k\tau_{|\bold {m_i}|+1}\rangle_g\\
=RHS-\sum_{\substack{\bold L+\bold{L'}=\bold b\\\bold L\neq \bold
0}}(-1)^{||\bold L||}\binom{\bold b}{\bold L}\frac{(2d_1+2|\bold
L|+1)!!}{(2|\bold L|+1)!!} \langle\kappa(\bold L')\tau_{d_1+|\bold
L|}\prod_{j=2}^n\tau_{d_j}\rangle_g\\
=RHS-LHS+(2d_1+1)!!\langle\prod_{j=1}^n\tau_{d_j}\kappa(\bold
b)\rangle_g.
\end{multline*}

In the third equation, only the quadratic term needs a careful
verification. So we have proved $RHS=LHS$.

\bigskip
We will see that Theorem 1.2 follows from Theorem 1.1 and Lemma 2.1.
\bigskip

\noindent{\bf Proof of Theorem 1.2} \smallskip

Let $$F(\bold b,d_1)=\frac{(2d_1+1)!!}{\bold
b!}\langle\prod_{j=1}^n\tau_{d_j}\kappa(\bold b)\rangle_g$$ and
\begin{multline*}
G(\bold b,d_1)=\sum_{j=2}^n\frac{(2(|\bold L|+d_1+d_j)-1)!!}{\bold
b!(2d_j-1)!!}\langle\kappa(\bold{b})\tau_{d_1+d_j-1}\prod_{i\neq
1,j}\tau_{d_i}\rangle_g\\
+\frac{1}{2}\sum_{r+s=d_1-2}\frac{(2r+1)!!(2s+1)!!}{\bold b!}\langle\kappa(\bold{b})\tau_r\tau_s\prod_{i=2}^n\tau_{d_i}\rangle_{g-1}\\
+\frac{1}{2}\sum_{\substack{\bold{e}+\bold{f}=\bold b\\I\coprod
J=\{2,\dots,n\}}}\sum_{r+s=d_1-2}\frac{(2r+1)!!(2s+1)!!}{\bold e!\bold f!}\\
\times \langle\kappa(\bold{e})\tau_r\prod_{i\in
I}\tau_{d_i}\rangle_{g'}\langle\kappa(\bold{f})\tau_s\prod_{i\in
J}\tau_{d_i}\rangle_{g-g'}.
\end{multline*}

Note that Theorem 1.1 is just
$$\sum_{\bold L+\bold{L'}=\bold b} \frac{(-1)^{||\bold L||}}{\bold L!(2|\bold{L}|+1)!!} F(\bold{L'},d_1+|\bold L|)=G(\bold b,d_1).$$

By Lemma 2.1, we have
$$F(\bold b,d_1)=\sum_{\bold L+\bold{L'}=\bold b}\frac{\alpha_{\bold L}}{\bold{L}!}G(\bold{L'},d_1+|\bold L|),$$
which is just the result we want.

\vskip 30pt
\section{Higher Weil-Petersson volumes}
By applying Lemma 2.2 as in the proof of Theorem 1.1, we may
generalize recursions of pure $\psi$ classes to recursions including
both $\psi$ and $\kappa$ classes.

First we have the following generalization of the string and
dilation equations.
\begin{proposition}
For $\bold b\in N^\infty$, $n\geq 0$ and $d_j\geq0$,
$$\sum_{\bold L+\bold L'=\bold
b}(-1)^{||\bold L||}\binom{\bold b}{\bold L}\langle\tau_{|\bold
L|}\prod_{j=1}^n\tau_{d_j}\kappa(\bold
L')\rangle_g=\sum_{j=1}^n\langle\tau_{d_j-1}\prod_{i\neq
j}\tau_{d_i}\kappa(\bold b)\rangle_g,$$ and
$$\sum_{\bold L+\bold L'=\bold b}(-1)^{||\bold
L||}\binom{\bold b}{\bold L}\langle\tau_{|\bold
L|+1}\prod_{j=1}^n\tau_{d_j}\kappa(\bold
L')\rangle_g=(2g-2+n)\langle\prod_{j=1}^n\tau_{d_j}\kappa(\bold
b)\rangle_g.$$
\end{proposition}
\begin{proof}
The first identity follows by taking $d_1=0$ in Theorem 1.1. For the
second identity, we have
\begin{multline*}
\langle\prod_{j=1}^n\tau_{d_j}\tau_1\kappa(\bold b)\rangle_g\\
=\sum_{k\geq0}\sum_{\substack{\bold{m_1}+\cdots+\bold{m_k}=\bold
b\\\bold{m_i}\neq 0}}\frac{(-1)^{||\bold b||-k}}{k!}\binom{\bold
b}{\bold{m_1\dots,m_k}}
\langle\tau_1\prod_{j=1}^n\tau_{d_j}\prod_{j=1}^k\tau_{|\bold
m_j|+1}\rangle_g\\
=(2g+n-2)\langle\prod_{j=1}^n\tau_{d_j}\kappa(\bold b)\rangle_g\qquad\qquad\qquad\qquad\qquad\qquad\qquad\qquad\\
+\sum_{k\geq0}\sum_{\substack{\bold
L+\bold{m_1}+\cdots+\bold{m_k}=\bold b\\\bold L\neq0,\bold{m_i}\neq
0}}\frac{(-1)^{||\bold b||-k-1}}{k!}\binom{\bold
b}{\bold{L,m_1\dots,m_k}} \langle\tau_{|\bold
L|+1}\prod_{j=1}^k\tau_{|\bold
m_j|+1}\prod_{j=1}^n\tau_{d_j}\rangle_g.
\end{multline*}
Subtracting the last term from each side, we have proved the second
identity.
\end{proof}

For the particular case $\bold b=(m,0,0,\dots)$, Proposition 3.1 has
been proved by Norman Do and Norbury \cite{DN} in their study of the
intermediary moduli spaces consisting of hyperbolic surfaces with a
cone point of a specified angle.

We need the following results from \cite{Ar-Co}.

\begin{lemma} Let $\pi_{n+1}: \overline{\sM}_{g,n+1}\longrightarrow
\overline{\sM}_{g,n}$ be the morphism that forgets the last marked
point.
\begin{enumerate}
\item[i)] $\pi_{n*}(\psi_1^{a_1}\cdots\psi_{n-1}^{a_{n-1}}\psi_n^{a_{n}+1})=
\psi_1^{a_1}\cdots\psi_{n-1}^{a_{n-1}}\kappa_{a_n}$\qquad
\text{for}\quad $a_j\geq0$;

\item[ii)] $\kappa_a=\pi^*_{n+1}(\kappa_a)+\psi^a_{n+1}$\qquad
\text{on}\quad $\overline{\sM}_{g,n+1}$;

\item[iii)] $\kappa_0=2g-2+n$\qquad
\text{on}\quad $\overline{\sM}_{g,n}$.
\end{enumerate}
\end{lemma}

We have the following generalization of a recursion formula from the
Witten-Kontsevich theorem corresponding to the first equation in the
KdV hierarchy (see Theorem 1.2 of \cite{LX2}).
\begin{proposition}
Let $\bold b\in N^\infty$ and $n\geq 0$. Then
\begin{multline}
\langle\tau_0\tau_1\prod_{j=1}^n\tau_{d_j}\kappa(\bold{
b})\rangle_g=\frac{1}{12}\langle\tau_0^4\prod_{
j=1}^n\tau_{d_j}\kappa(\bold b)\rangle_{g-1}\\
+\frac{1}{2}\sum_{\substack{\bold L+\bold{L'}=\bold
b\\\underline{n}=I\coprod J}}\binom{\bold b}{\bold
L}\langle\tau_0^2\prod_{i\in I}\tau_{d_i}\kappa(\bold
L)\rangle_{g'}\langle\tau_0^2\prod_{i\in
J}\tau_{d_i}\kappa(\bold{L'})\rangle_{g-g'}.
\end{multline}
\end{proposition}

Now we give a proof of Theorem 1.3. Let LHS and RHS denote the left
and right hand side of Proposition 3.3 respectively. Taking $d_j=0$
and applying Lemma 3.2, we have
\begin{multline*}
LHS=\int_{\overline{\sM}_{g,n+1}}\pi_{n+2*}\left(\psi_{n+2}\prod_{i\geq
1}(\pi_{n+2}^*\kappa_i+\psi_{n+2}^i)^{b(i)}\right) \\=\sum_{\bold
L+\bold{L'}=\bold b}\binom{\bold b}{\bold
L}\langle\tau_0^{n+1}\kappa(\bold L)\kappa_{|\bold{L'}|}\rangle_g\\
=\big((2g-1+n)+||\bold b||\big)\langle\tau_0^{n+1}\kappa(\bold
b)\rangle_g+\sum_{\substack{\bold L+\bold{L'}=\bold
b\\||\bold{L'}||\geq2}}\binom{\bold b}{\bold
L}\langle\tau_0^{n+1}\kappa(\bold L)\kappa_{|\bold{L'}|}\rangle_g
\end{multline*} and
\begin{multline*}
RHS=\frac{1}{12}\langle\tau_0^{n+4}\kappa(\bold b)\rangle_{g-1}
+\frac{1}{2}\sum_{\bold L+\bold{L'}=\bold b}\sum_{r+s=n}\binom{\bold
b}{\bold L}\binom{n}{r}\langle\tau_0^{r+2}\kappa(\bold
L)\rangle_{g'}\langle\tau_0^{s+2}\kappa(\bold{L'})\rangle_{g-g'}\\
=\frac{1}{12}\langle\tau_0^{n+4}\kappa(\bold b)\rangle_{g-1}+\frac{1}{2}\sum_{\substack{\bold L+\bold{L'}=\bold b\\
\bold{L}\neq\bold 0, \bold{L'}\neq\bold 0}}\sum_{r+s=n}\binom{\bold
b}{\bold L}\binom{n}{r}\langle\tau_0^{r+2}\kappa(\bold
L)\rangle_{g'}\langle\tau_0^{s+2}\kappa(\bold{L'})\rangle_{g-g'}\\
+n\langle\tau_0^{n+1}\kappa(\bold b)\rangle_g.
\end{multline*}
So Theorem 1.3 follows from $LHS=RHS$.

By further expanding the term $V_{g-1,n+3}(\kappa(\bold b))$ in
Theorem 1.3, we get
\begin{multline*}
V_{g,n}(\kappa(\bold b))=\delta_{||\bold
b||,0}+\frac{1}{24^gg!}\delta_{||\bold
b||,1}+\sum_{h=0}^g\frac{(2h-3+||\bold b||)!!}{12^{g-h}(2g-1+||\bold
b||)!!}\times\\ \left(
\frac{1}{2}\sum_{\substack{\bold L+\bold{L'}=\bold b\\
\bold{L}\neq\bold 0, \bold{L'}\neq\bold
0}}\sum_{r+s=n-1+3(g-h)}\binom{\bold b}{\bold
L}\binom{n-1+3(g-h)}{r}V_{h',r+2}(\kappa(\bold
L))V_{h-h',s+2}(\kappa(\bold {L'}))\right.\\\left.
-\sum_{\substack{\bold L+\bold{L'}=\bold b\\
||\bold{L'}||\geq 2}}\binom{\bold b}{\bold
L}V_{h,n+3(g-h)}(\kappa(\bold L)\kappa_{|\bold{L'}|})\right).
\end{multline*}

The following proposition is a generalization of a recursion formula
proved in Proposition 2.6 of \cite{LX2}.
\begin{proposition}
Let $\bold b\in N^\infty$, $n\geq 0$ and $r\geq0$. Then
\begin{multline}
\langle\tau_1\tau_r\prod_{j=1}^n\tau_{d_j}\kappa(\bold{
b})\rangle_g=(2r+3)\langle\tau_0\tau_{r+1}\prod_{j=1}^n\tau_{d_j}\kappa(\bold
b)\rangle_{g}-\frac{1}{6}\langle\tau_0^3\tau_r\prod_{j=1}^n\tau_{d_j}\kappa(\bold b)\rangle_{g-1}\\
-\sum_{\substack{\bold L+\bold{L'}=\bold b\\\underline{n}=I\coprod
J}}\binom{\bold b}{\bold L}\langle\tau_0\tau_r\prod_{i\in
I}\tau_{d_i}\kappa(\bold L)\rangle_{g'}\langle\tau_0^2\prod_{i\in
J}\tau_{d_i}\kappa(\bold{L'})\rangle_{g-g'}.
\end{multline}
\end{proposition}
Let LHS and RHS denote the left and right hand side of Proposition
3.4 respectively. Taking $r=1$ and $n=0$, we have
\begin{multline*}
LHS=\int_{\overline{\sM}_{g,1}}\pi_{2*}\left(\psi_1\psi_2\prod_{i\geq
1}(\pi_{2}^*\kappa_i+\psi_{2}^i)^{b(i)}\right)\\
=\sum_{\bold L+\bold{L'}=\bold b}\binom{\bold b}{\bold
L}\int_{\overline{\sM}_{g,1}}\psi_1\kappa(\bold
L)\kappa_{|\bold{L'}|}\\
=(||\bold b||+2g-1)\int_{\overline{\sM}_{g,1}}\psi_1\kappa(\bold
b)+\sum_{\substack{\bold L+\bold{L'}=\bold b\\||\bold{L'}||\geq
2}}\int_{\overline{\sM}_{g,1}}\kappa(\bold L)\kappa_{|\bold{L'}|}
\\=\big((2g-1)(2g-2)+(4g-3)||\bold b||+||\bold
b||^2\big)V_{g}(\kappa(\bold b))\\
+(2g-1+||\bold b||)\sum_{\substack{\bold L+\bold{L'}=\bold
b\\||\bold{L'}||\geq 2}}\binom{\bold b}{\bold L}V_{g}(\kappa(\bold
L)\kappa_{|\bold{L'}|})\\
+\sum_{\substack{\bold L+\bold{L'}=\bold b\\||\bold{L'}||\geq
2}}\binom{\bold b}{\bold L}\sum_{\bold e+\bold f=\bold
L+\bm{\delta}_{|\bold{L'}|}}\binom{\bold
L+\bm{\delta}_{|\bold{L'}|}}{\bold e} V_{g}(\kappa(\bold
e)\kappa_{|\bold f|}).
\end{multline*}
and similarly,
\begin{multline*}
RHS=5\sum_{\bold L+\bold{L'}=\bold b}\binom{\bold b}{\bold
L}V_{g,1}(\kappa(\bold L)\kappa_{|\bold{L'}|+1})
-\frac{1}6{}\sum_{\bold L+\bold{L'}=\bold b}\binom{\bold b}{\bold
L}V_{g-1,3}(\kappa(\bold L)\kappa_{|\bold{L'}|})\\-\sum_{\bold
L+\bold{e}+\bold{f}=\bold b}\binom{\bold b}{\bold{
L,e,f}}V_{g',1}(\kappa_{|\bold L|}\kappa(\bold
e))V_{g-g',2}(\kappa(\bold f)).
\end{multline*}
So we have proved Theorem 1.4.

\vskip 30pt
\section{Virasoro constraints and the KdV hierarchy}
In this section, we follow the arguments of Mulase and Safnuk
\cite{MS} to study properties of generating functions of
intersections of $\psi$ and $\kappa$ classes using Theorems 1.1 and
1.2.

Let $\bold s:=(s_1,s_2,\dots)$ and $\bold t:=(t_0,t_1,t_2,\dots)$,
we introduce the following generating function
$$G(\bold s,\bold t):=\sum_{g}\sum_{\bold m,\bold n}\langle\kappa_1^{m_1}\kappa_2^{m_2}\cdots\tau_0^{n_0}\tau_1^{n_1}\cdots\rangle_g\frac{\bold s^{\bold m}}{\bold m!}\prod_{i=0}^\infty\frac{t_i^{n_i}}{n_i!},$$
where $\bold s^{\bold m}=\prod_{i\geq1}s_i^{m_i}$.

Propositions 3.3 and 3.4 can be reformulated in terms of
differential operators.
\begin{proposition} Let $r\geq0$. Then we have
$$\frac{\partial^2{G}}{\partial t_0\partial t_1}=\frac{1}{12}\frac{\partial^4G}{\partial t_0^4}
+\frac{1}{2}\frac{\partial^2G}{\partial
t_0^2}\frac{\partial^2G}{\partial t_0^2}
$$ and
$$\frac{\partial^2G}{\partial t_1\partial t_r}=(2r+3)\frac{\partial^2G}{\partial t_0\partial t_{r+1}}
-\frac{1}{6}\frac{\partial^4G}{\partial t_0^3\partial
t_r}-\frac{\partial^2G}{\partial t_0\partial
t_r}\frac{\partial^2G}{\partial t_0^2}.
$$
\end{proposition}

We define $\beta_{\bold L}=\alpha_{\bold L}/\bold L!$ where
$\alpha_{\bold L}$ are the same constants in Theorem 1.2. We
introduce the following family of differential operators for $k\geq
-1$,
\begin{multline}\label{VKhat}
 \hat V_{k} = -\frac{(2k+3)!! }{2}\frac{\partial}{\partial t_{k+1}}
    + \delta_{k,-1}(\frac{t_0^2}{4} + \frac{s_1}{48}) + \frac{\delta_{k,0}}{16} \\
    + \frac{1}{2}\sum_{\bold L}\sum_{j=0}^{\infty} \frac{(2(|\bold L|+j+k)+1)!! }{(2j-1)!! }\beta_{\bold L}\bold s^{\bold L} t_j
    \frac{\partial}{\partial t_{|\bold L|+j+k}} \\
    + \frac{1}{4}\sum_{\bold L} \sum_{\substack{d_1+d_2=\\|\bold L|+k-1 }}(2d_1 +1)!!
        (2d_2+1)!! \beta_{\bold L}\bold s^{\bold L} \frac{\partial^2}{\partial t_{d_1} \partial t_{d_2}}.
\end{multline}

\begin{theorem}
We have $\hat V_k\exp(G)=0$ for $k\geq-1$ and
\begin{equation*}
[\hat V_n,\hat V_m]=(n-m)\sum_{\bold L}\beta_{\bold L}\bold s^{\bold
L}\hat V_{n+m+|\bold L|}.
\end{equation*}
\end{theorem}
\begin{proof}
Note that the termination cases of the recursion formula in Theorem
1.2 are
$$\langle\tau_0\kappa_1\rangle_1=\frac{1}{24},\qquad
\langle\tau_0^3\rangle_0=1,\qquad
\langle\tau_1\rangle_1=\frac{1}{24}.$$ So $\hat V_k\exp(G)=0$ for
$k\geq-1$ is just a restatement of Theorem 1.2.

One may check directly that
$$[\hat V_n,\hat V_m]=(n-m)\sum_{\bold L}\beta_{\bold L}\bold s^{\bold L}\hat V_{n+m+|\bold L|}.$$
\end{proof}

The following constants are inverse to $\beta_{\bold L}$,
$$\gamma_{\bold L}:=\frac{(-1)^{||\bold L||}}{\bold L!(2|\bold L|+1)!!}.$$
Define a new family of differential operators $V_k$ for $k\geq -1$
by
 \begin{multline} \label{VK}
    V_k = -\frac{1}{2} \sum_{\bold L} (2(|\bold L|+k)+3)!! \gamma_{\bold L} \bold
    s^{\bold L}
        \frac{\partial }{\partial t_{|\bold L|+k+1} }
        + \frac{1}{2} \sum_{j=0}^{\infty} \frac{(2(j+k)+1)!! }{(2j-1)!! } t_j
        \frac{\partial }{\partial t_{j+k} } \\
    + \frac{1}{4} \sum_{d_1 + d_2 = k-1}
        (2d_1 + 1)!! (2d_2 + 1)!! \frac{\partial^2 }{\partial t_{d_1} \partial t_{d_2}}
        + \frac{\delta_{k,-1}t_0^2}{4} + \frac{\delta_{k,0} }{16},
 \end{multline}
Theorem 1.1 implies $V_k\exp(G)=0$. We now prove that the operators
$V_k$ satisfy the Virasoro relations
$$[V_n,V_m]=(n-m)V_{n+m}.$$

Introduce new variables
$$T_{2i+1}:=\frac{t_i}{(2i+1)!!},\quad i\geq0$$
which transform the operators
$\hat V_k$ into
\begin{align*}
    \hat V_{k} =& -\frac{1}{2}\frac{\partial }{\partial T_{2k+3} }
    + \delta_{k,-1} (\frac{t_0^2}{4} + \frac{s_1}{48}) + \frac{\delta_{k,0} }{16 } \\
    & + \frac{1}{2} \sum_{\bold L}\sum_{j=0}^{\infty} (2j+1) \beta_{\bold L} \bold s^{\bold L} T_{2j+1} \frac{\partial}{\partial   T_{2(|\bold L|+j+k) + 1}} \\
    & + \frac{1}{4} \sum_{\bold L} \sum_{\substack{d_1+d_2 = \\ |\bold L|+k-1 }} \beta_{\bold L} \bold s^{\bold L} \frac{\partial^2}{\partial T_{2d_1 +1} \partial T_{2d_2 +1} }.
\end{align*}
Define operators $J_p$ for $p\in\mathbb Z$ by
\begin{equation*}
    J_p = \begin{cases}
        (-p) T_{-p} & \text{if $p<0$}, \\
        \frac{\partial}{\partial T_p} & \text{if $p>0$}.
    \end{cases}
\end{equation*}
Then
\begin{align*}
    \hat V_k &= -\frac{1}{2} J_{2k+3} + \sum_{\bold L}
        \beta_{\bold L}\bold s^{\bold L} E_{k+|\bold L|}, \\
    \intertext{where}
    E_k &= \frac{1}{4}\sum_{p\in\mathbb Z} J_{2p+1} J_{2(k-p) - 1}
        + \frac{\delta_{k,0}}{16}.
\end{align*}

It's not difficult to see that
$$
V_k=\sum_{\bold L}\gamma_{\bold L}\bold s^{\bold L}\hat V_{k+|\bold
L|}\notag=-\frac{1}{2}\sum_{\bold L}\gamma_{\bold L}\bold s^{\bold
L}J_{2k+2|\bold L|+3}+E_k.
$$

\begin{theorem}
The operators $V_k$, $k\geq-1$ satisfy the Virasoro relations
$$[V_n,V_m]=(n-m)V_{n+m}.$$
\end{theorem}
\begin{proof}
Since
\begin{align*}
E_k=&\frac{1}{2} \sum_{j=0}^{\infty} \frac{(2(j+k)+1)!! }{(2j-1)!! }
t_j\frac{\partial }{\partial t_{j+k} } \\
&+\frac{1}{4} \sum_{d_1 + d_2 = k-1}(2d_1 + 1)!! (2d_2 + 1)!!
\frac{\partial^2 }{\partial t_{d_1} \partial t_{d_2}}
+\frac{\delta_{k,-1}t_0^2}{4} + \frac{\delta_{k,0} }{16}.
\end{align*}
We can check directly that
$$[E_n,E_m]=(n-m)E_{n+m},\qquad [J_{2k+3},E_m]=\frac{2k+3}{2}J_{2(k+m)+3}.$$
So we have
\begin{align*}
[V_n, V_m] &= \bigl[-\frac{1}{2} \sum_{\bold L}\gamma_{\bold L}\bold
s^{\bold L} J_{2(n+|\bold L|)+3 } + E_n, -\frac{1}{2}\sum_{\bold L}
\gamma_{\bold L}\bold s^{\bold L} J_{2(m+|\bold L|)+3 } + E_m\bigr] \\
&= -\frac{1}{2} \sum_{\bold L} \gamma_{\bold L}\bold s^{\bold L}
\left(\bigl[J_{2(n+|\bold L|)+3 }, E_m] + [E_n, J_{2(m+|\bold L|)+3 }\bigr]\right)+[E_n,E_m] \\
&=-\frac{1}{2} \sum_{\bold L}\gamma_{\bold L}\bold s^{\bold L}(n-m)J_{2(n+m+|\bold L|)+3}+(n-m)E_{n+m}\\
        &= (n-m) V_{n+m }.
\end{align*}
\end{proof}

Now we recall the KdV hierarchy, which is the following hierarchy of
differential equations for $n\geq 1$,
$$\frac{\partial U}{\partial t_n}=\frac{\partial}{\partial t_0}R_{n+1},$$
where $R_n$ are polynomials in $U,\partial U/\partial t_0,\partial^2
U/\partial t_0^2,\dots$, which is defined recursively by
$$R_1=U,\qquad \frac{\partial R_{n+1}}{\partial t_0}=\frac{1}{2n+1}\left(\frac{\partial U}{\partial t_0}R_n+2U\frac{\partial R_n}{\partial t_0}+
\frac{1}{4}\frac{\partial^3}{\partial t_0^3}R_n\right).$$ In
particular, it is easy to see that
$$R_2=\frac{1}{2}U^2+\frac{1}{12}\frac{\partial^2U}{\partial
t_0^2},$$ so the first equation in the KdV hierarchy is the
classical KdV equation
$$\frac{\partial U}{\partial t_1}=U\frac{\partial U}{\partial t_0}+\frac{1}{12}\frac{\partial^3 U}{\partial t_0^3}.$$

The Witten-Kontsevich theorem  \cite{Wi,Ko}
 states that the generating function for $\psi$ class
intersections
$$F(t_0, t_1, \ldots)= \sum_{g} \sum_{\bold n} \langle\prod_{i=0}^\infty \tau_{i}^{n_i}\rangle_{g} \prod_{i=0}^\infty \frac{t_i^{n_i} }{n_i!
        }$$
is a $\tau$-function for the KdV hierarchy, i.e. $\partial^2
F/\partial t_0^2$ obeys all equations in the KdV hierarchy.

\begin{theorem} We have
\begin{equation}
G(\bold s,t_0,t_1,\dots)=F(t_0,t_1,t_2+p_2,t_3+p_3,\dots),
\end{equation}
where $p_k$ are polynomials in $\bold s$ given by
$$p_k=-\sum_{|\bold L|=k-1}(2|\bold L|+1)!!\gamma_{\bold L}\bold s^{\bold L}=\sum_{|\bold L|=k-1}\frac{(-1)^{||\bold L||-1}}{\bold L!}\bold
s^{\bold L}.$$ In particular, for any fixed values of $\bold s$,
$G(\bold s,\bold t)$ is a $\tau$-function for the KdV hierarchy.
\end{theorem}
\begin{proof}
The change of variables
$$
    \tilde{t}_i = \begin{cases}
        t_i & \text{for $i=0,1$ }, \\
        t_i - \sum_{|\bold L|=i-1}(2|\bold L|+1)!!\gamma_{\bold L}\bold s^{\bold L} & \text{otherwise,}
    \end{cases}
$$
transforms the operators $V_k$ of (\ref{VK}) into
\begin{multline*}
    V_k = -\frac{1}{2}  (2k+3)!!
        \frac{\partial }{\partial \tilde{t}_{k+1} }
        + \frac{1}{2} \sum_{j=0}^{\infty} \frac{(2(j+k)+1)!! }{(2j-1)!! } \tilde{t}_j
        \frac{\partial }{\partial \tilde{t}_{j+k} } \\
    + \frac{1}{4} \sum_{d_1 + d_2 = k-1}
        (2d_1 + 1)!! (2d_2 + 1)!! \frac{\partial^2 }{\partial \tilde{t}_{d_1}
        \partial \tilde{t}_{d_2}}
        + \frac{\delta_{k,-1}\tilde{t}_0^2}{4} + \frac{\delta_{k,0}
        }{16},
 \end{multline*}
which is just the operator obtained by setting $\bold s=\bold0$ in
$\hat V_k$ of (\ref{VKhat}). Since Virasoro constraints uniquely
determine the generating functions $G(\bold s, t_0,t_1,\dots)$ and
$F(t_0,t_1,\dots)$, we have for any fixed values of $\bold s$,
$$G(\bold s, t_0,t_1,t_2,\dots)=F(\tilde{t}_0,\tilde{t}_1,\tilde{t}_2,\dots).$$
So we have proved the theorem.
\end{proof}

Theorem 4.4 can also be proved directly by applying Lemma 2.2, as
discussed in \cite{MZ}.

\vskip 30pt
\section{Tautological constants of Hodge integrals}

The results in this section can be applied to study Faber's perfect
pairing conjecture \cite{Fab} and its generalizations.

Let $\sM_{g,n}^{rt}$ be the moduli space of ``curves with rational
tails''(i.e. with dual graph with a vertex of genus $g$). Let
$\sM_{g,n}^{ct}$ be the moduli space of ``curves of compact type'',
(i.e. with dual graph with no loops). Hence
$$\sM_{g,n}^{rt} \subset \sM_{g,n}^{ct} \subset \overline{\sM}_{g,n}.$$

\begin{conjecture} (Faber, Hain, Looijenga, Pandharipande, et al.) The space
$\overline{\sM}_{g,n}$ (resp.\ $\sM_{g,n}^{rt}$, $\sM_{g,n}^{ct}$)
``behaves like'' a complex variety of dimension $D = 3g-3+n$ (resp.
 $g-2+n$, $2g-3+n$). More precisely, its tautological ring $R^*$
has the following properties.
\begin{itemize}
\item {\rm Socle statement:} $R^i = 0$ for $i>D$, $R^D \cong \mathbb Q$, and
\item {\rm Perfect pairing statement:} for $0 \leq i \leq D$, the natural map $R^i \times R^{D-i} \rightarrow
R^D$ is a perfect pairing.
\end{itemize}
\end{conjecture}

The socle statement has been proved by Graber and Vakil \cite{GV}.
While the perfect paring statement is still open.

By the above conjecture, tautological relations in $\sM_{g,n}^{rt}$
and $\sM_{g,n}^{ct}$ are determined respectively by the following
linear functionals, called intersection pairings.
\begin{align*}
 R^i(\sM_{g,n}^{rt})\times R^{g-2+n-i}(\sM_{g,n}^{rt})
&\longrightarrow \mathbb Q\\
(u,v) &\longmapsto
\int_{\overline{\sM}_{g,n}}uv\lambda_g\lambda_{g-1},
\end{align*}
and
\begin{align*}
R^i(\sM_{g,n}^{ct})\times R^{2g-3+n-i}(\sM_{g,n}^{ct})
&\longrightarrow \mathbb Q\\
(u,v) &\longmapsto \int_{\overline{\sM}_{g,n}}uv\lambda_g.
\end{align*}

Since tautological classes are represented by linear combinations of
decorated stable graphs, the computation of intersection pairings
will eventually reduce to the following integrals
\begin{align*}
&\int_{\overline{\sM}_{g,n}}\kappa_{b_1}\cdots\kappa_{b_k}\psi_1^{d_1}\cdots\psi_n^{d_n}\lambda_g\lambda_{g-1},\\
&\int_{\overline{\sM}_{g,n}}\kappa_{b_1}\cdots\kappa_{b_k}\psi_1^{d_1}\cdots\psi_n^{d_n}\lambda_g.
\end{align*}

Commonly, one would compute the above integrals by first eliminating
$\kappa$ classes, then applying the $\lambda_g\lambda_{g-1}$ theorem
or the $\lambda_g$ theorem.

Now we present more efficient recursion formulae computing these
integrals, their patterns may well give some implications of the
perfect pairing conjectures.

From degree 0 Virasoro constraints for a surface, Getzler and
Pandharipande \cite{Ge-Pa} obtained the following recursion.
\begin{lemma}\cite{Ge-Pa}
Let $d,d_0\geq0$ and $d_j\geq 1$ for $j\geq1$.
\begin{align*}
\langle\tau_{d}\tau_{d_0}\prod_{j=1}^n\tau_{d_j}\mid\lambda_g\lambda_{g-1}\rangle_g=\frac{(2d+2d_0-1)!!}{(2d-1)!!(2d_0-1)!!}
\langle\tau_{d_0+d-1}\prod_{j=1}^n\tau_{d_j}\mid\lambda_g\lambda_{g-1}\rangle_g\\
+\sum_{j=1}^n\frac{(2d+2d_j-3)!!}{(2d-1)!!(2d_j-3)!!}\langle\tau_{d_0}\tau_{d_j+d-1}\prod_{i\neq
j}\tau_{d_i}\mid\lambda_g\lambda_{g-1}\rangle_g
\end{align*}
\end{lemma}

Lemma 5.2 has the following generalization.
\begin{theorem}
Let $\bold b\in N^{\infty}$, $d,d_0\geq0$ and $d_j\geq 1$ for
$j\geq1$. Then
\begin{multline*}
\sum_{\bold L+\bold{L'}=\bold b}(-1)^{||\bold L||}\binom{\bold
b}{\bold L}\frac{(2d+2|\bold L|-1)!!}{(2|\bold
L|-1)!!}\langle\tau_{d+|\bold
L|}\tau_{d_0}\prod_{j=1}^n\tau_{d_j}\kappa(\bold{L'})\mid\lambda_g\lambda_{g-1}\rangle_g\\
=\frac{(2d+2d_0-1)!!}{(2d_0-1)!!}
\langle\tau_{d_0+d-1}\prod_{j=1}^n\tau_{d_j}\kappa(\bold{b})\mid\lambda_g\lambda_{g-1}\rangle_g\\
+\sum_{j=1}^n\frac{(2d+2d_j-3)!!}{(2d_j-3)!!}\langle\tau_{d_0}\tau_{d_j+d-1}\prod_{i\neq
j}\tau_{d_i}\kappa(\bold{b})\mid\lambda_g\lambda_{g-1}\rangle_g
\end{multline*}
and
\begin{multline*}
\langle\tau_{d}\tau_{d_0}\prod_{j=1}^n\tau_{d_j}\kappa(\bold
b)\mid\lambda_g\lambda_{g-1}\rangle_g\\
=\sum_{\bold L+\bold{L'}=\bold b}\gamma_{\bold L}\binom{\bold
b}{\bold L}\frac{(2d+2d_0+2|\bold L|-1)!!}{(2d-1)!!(2d_0-1)!!}
\langle\tau_{d_0+d+|\bold L|-1}\prod_{j=1}^n\tau_{d_j}\kappa(\bold{L'})\mid\lambda_g\lambda_{g-1}\rangle_g\\
+\sum_{\bold L+\bold{L'}=\bold b}\sum_{j=1}^n\gamma_{\bold
L}\binom{\bold b}{\bold L}\frac{(2d+2d_j+2|\bold
L|-3)!!}{(2d-1)!!(2d_j-3)!!}\langle\tau_{d_0}\tau_{d_j+d+|\bold
L|-1}\prod_{i\neq
j}\tau_{d_i}\kappa(\bold{L'})\mid\lambda_g\lambda_{g-1}\rangle_g
\end{multline*}
where $\gamma_{\bold L}\in\mathbb Q$ can be determined recursively
from the following formula
$$\sum_{\bold L+\bold{L'}=\bold b}\frac{(-1)^{||\bold L||}\gamma_{\bold L}}{\bold L!\bold{L'}!(2|\bold{L'}|-1)!!}=0,\qquad \bold b\neq0,$$
with the initial value $\gamma_{\bold 0}=1$.

\end{theorem}

\begin{corollary}In Theorem 5.3, we have
$$\gamma_l=\frac{E_l}{(2l-1)!!},\quad \gamma(\underbrace{0,\dots,0,1}_{l})=\frac{1}{(2l-1)!!}$$
where $E_l$ are the Euler numbers that satisfy
$$\sec x=\frac{1}{\cos x}=\sum_{k=0}^\infty\frac{E_k}{(2k)!}x^{2k}=1+\frac{1}{2!}x^2+\frac{5}{4!}x^4+\frac{61}{6!}x^6
+\frac{1385}{8!}x^8+\frac{50521}{10!}x^{10}+\cdots.$$
\end{corollary}
\begin{proof}
We have
$$\cos(\sqrt{2}x)=\sum_{k=0}^\infty\frac{(-1)^k}{k!(2k-1)!!}x^{2k},$$
by Theorem 5.3,
$$\sec(\sqrt{2}x)=\sum_{k=0}^\infty\frac{\gamma_k}{k!}x^{2k}.$$
So we get the formula of $\gamma_l$.
\end{proof}

The following recursion follows from degree $0$ Virasoro constraints
for a curve.
\begin{lemma}\cite{Ge-Pa}
Let $d,d_0\geq0$ and $d_j\geq 1$ for $j\geq1$.
\begin{align*}
\langle\tau_{d}\tau_{d_0}\prod_{j=1}^n\tau_{d_j}\mid\lambda_g\rangle_g=\binom{d+d_0}{d_0}
\langle\tau_{d_0+d-1}\prod_{j=1}^n\tau_{d_j}\mid\lambda_g\rangle_g\\
+\sum_{j=1}^n\binom{d_j+d-1}{d_j-1}\langle\tau_{d_0}\tau_{d_j+d-1}\prod_{i\neq
j}\tau_{d_i}\mid\lambda_g\rangle_g,
\end{align*}
\end{lemma}

Lemma 5.5 has the following generalization.
\begin{theorem}
Let $\bold b\in N^{\infty}$, $d,d_0\geq0$ and $d_j\geq 1$ for
$j\geq1$.
\begin{multline*}
\sum_{\bold L+\bold{L'}=\bold b}\binom{\bold b}{\bold
L}(-1)^{||\bold
L||}\frac{(d+|\bold L|)!}{|\bold L|!}\langle\tau_{d}\tau_{d_0}\prod_{j=1}^n\tau_{d_j}\kappa(\bold{L'})\mid\lambda_g\rangle_g\\
=\frac{(d+d_0)!}{d_0!}
\langle\tau_{d_0+d-1}\prod_{j=1}^n\tau_{d_j}\kappa(\bold{b})\mid\lambda_g\rangle_g\\
+\sum_{j=1}^n\frac{(d_j+d-1)!}{(d_j-1)!}\langle\tau_{d_0}\tau_{d_j+d-1}\prod_{i\neq
j}\tau_{d_i}\kappa(\bold{b})\mid\lambda_g\rangle_g
\end{multline*}
and
\begin{multline*}
\langle\tau_{d}\tau_{d_0}\prod_{j=1}^n\tau_{d_j}\kappa(\bold
b)\mid\lambda_g\rangle_g\\
=\sum_{\bold L+\bold{L'}=\bold b}\gamma_{\bold L}\binom{\bold
b}{\bold L}\frac{(d+d_0+|\bold L|)!}{d_0!d!}
\langle\tau_{d_0+d+|\bold L|-1}\prod_{j=1}^n\tau_{d_j}\kappa(\bold{L'})\mid\lambda_g\rangle_g\\
+\sum_{\bold L+\bold{L'}=\bold b}\sum_{j=1}^n\gamma_{\bold
L}\binom{\bold b}{\bold L}\frac{(d_j+d+|\bold
L|-1)!}{(d_j-1)!d!}\langle\tau_{d_0}\tau_{d_j+d+|\bold
L|-1}\prod_{i\neq
j}\tau_{d_i}\kappa(\bold{L'})\mid\lambda_g\rangle_g
\end{multline*}
where $\gamma_{\bold L}\in\mathbb Q$ can be determined recursively
from the following formula
$$\sum_{\bold L+\bold{L'}=\bold b}\frac{(-1)^{||\bold L||}\gamma_{\bold L}}{\bold L!\bold{L'}!|\bold{L'}|!}=0,\qquad \bold b\neq0,$$
with the initial value $\gamma_{\bold 0}=1$.
\end{theorem}

We recall the definition of the Bessel functions of the first kind.
For the Bessel equations of order $\nu$
$$x^2y''+xy'+(x^2-\nu^2)y=0,$$
we have the following solutions
$$y=J_{\nu}(x)=\sum_{k=0}^\infty\frac{(-1)^k}{k!\Gamma(\nu+k+1)}\left(\frac{x}{2}\right)^{\nu+2k}.$$
These are called Bessel functions of the first kind of order
$\nu$.

\begin{corollary}In Theorem 5.6, we have
$$\gamma(\underbrace{0,\dots,0,1}_{l})=\frac{1}{l!}$$
and $\gamma_l$ is given by
$$\frac{1}{J_0(\sqrt{4x})}=\sum_{k=0}^{\infty}\frac{\gamma_k}{k!}x^k=1+x+\frac{3/2}{2!}x^2+\frac{19/6}{3!}x^3+
\frac{211/24}{4!}x^4+\frac{1217/40}{5!}x^5+\cdots,$$ where $J_0$
is the Bessel function of the first kind of order zero
$$J_0(x)=\sum_{k=0}^\infty\frac{(-1)^k}{4^k(k!)^2}x^{2k}.$$
\end{corollary}
\begin{proof}
The corollary follows easily from Theorem 5.6 and the following
$$J_0(\sqrt{4x})=\sum_{k=0}^\infty\frac{(-1)^k}{(k!)^2}x^k.$$
\end{proof}

It is interesting to notice that the Bessel function of the first
kind of order zero also appears in Manin and Zograf's work \cite{MZ}
on asymptotics for Weil-Petersson volumes.

$$ \ \ \ \ $$

\end{document}